\documentclass[12pt]{article}
\usepackage{amsmath, amssymb, amsfonts, amsthm}

\newcommand{\be}{\begin{eqnarray*}}
\newcommand{\ea}{\end{eqnarray*}}
\newcommand{\ket}[1]{| #1 \rangle}
\newcommand{\bra}[1]{\langle #1 |}
\newcommand{\RR}{\mathbb{R}}
\newcommand{\CC}{\mathbb{C}}

\newcommand{\NN}{\mathbb{N}}
\newtheorem{prop}{Proposition}
\newtheorem{axiom}{Definition}

\newtheorem{lem}{Lemma}
\newtheorem{thm}{Theorem}
\newcommand{\fig}{\begin{tikzpicture} \matrix [matrix of nodes, nodes={rectangle,draw} ,ampersand replacement=\& ]{
\, \& \, \& \, \\
\, \& \, \\
\, \\
\, \\
\, \\
};
\end{tikzpicture}}
\newcommand{\fg}{ 
\begin{tikzpicture} \matrix [matrix of nodes, nodes={rectangle,draw} ,ampersand replacement=\& ]{
5 \& 3 \& 2 \& 1 \\
4 \& 2 \& 1 \\
2 \& 1 \& 1 \\
1 \& 1  \\
};
\end{tikzpicture}}
\usepackage{tikz}
\usetikzlibrary{matrix}
\usetikzlibrary{positioning}
\begin{document}
\begin{center}
{\bf McMahon's Formula via Free Fermions}
\end{center}
\begin{abstract}
We give an elementary derivation of the vertex-operator derivation McMahon formula, counting all plane partitions of all size into a single generating function.  We fill in some details appearing in \cite{ORV} based on free fermions by defining an ``interlacing operator".  
\end{abstract}
\begin{paragraph}{} In this note, we prove McMahon's formula counting the number of plane partitions of given size.  Such a result is folklore and more general examples in be found in \cite{PS}.  \end{paragraph}
\begin{thm}The generating function for number of 3D partitions is
\[ \sum_{\text{3D partitions}} q^{\text{\# boxes}} =
 \prod_{n = 1}^\infty \frac{1}{(1 - q^n)^n} =  1 + q + 3q^2 + 6q^3 + \dots \]
 \end{thm}
\begin{paragraph}{}The sequence of coefficients $1,1,3,6,\dots$ is A000219 in Sloane's Encyclopedia of integer sequences, \cite{S}.  McMahon's identity
was proven in the 1910's and there are many
 combinatorial derivations.   In this note, we will follow the vertex operator proof given by Okounkov, Reshetikhin and Vafa in 2003, \cite{ORV}.  At the heart of the proof is the combinatorial notion of interlacing.
\end{paragraph}

\begin{axiom}We say $\mu \succ \nu$ are {\it interlacing} 2D partitions if $\mu_1 \geq \nu_1 \geq \mu_2 \geq \nu_2 \geq \dots$  \end{axiom}
\begin{prop} Every 3D partition has a decomposition into a finite sequence of interlacing young diagrams.\end{prop}
\begin{proof}
Rotate $45^\circ$. Slice the plane partition diagonally.
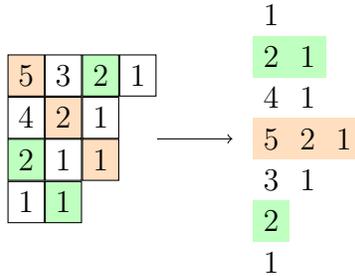
\begin{center}
\begin{tikzpicture}

\matrix [matrix of nodes, nodes={rectangle,draw}  ]{
 |[fill=orange!25]|5& 3 & |[fill=green!25]|2 & 1 \\
4 & |[fill=orange!25]|2 & 1 & \\
|[fill=green!25]|2 & 1 & |[fill=orange!25]|1 & \\
1 & |[fill=green!25]|1 &   & \\
};
\draw[->, xshift=1cm] (0, 0)--(1, 0);
\matrix [matrix of nodes, xshift=3cm]{
1 \\
|[fill=green!25]| 2 & |[fill=green!25]| 1 \\
4 & 1 \\
|[fill=orange!25]| 5 & |[fill=orange!25]| 2 & |[fill=orange!25]| 1 \\
3 & 1 \\
|[fill=green!25]| 2 \\
1 \\
};
\end{tikzpicture}
\end{center}

Starting from the upper left corner,  alternating 
down and to the right, we get a nonincreasing sequence of numbers from in the two diagonal slices.  The result is a sequence of interlacing partitions.  
 \end{proof}

{\bf The Hilbert Space of Young Diagrams }
\begin{paragraph}{}We have quantized the set of Young diagrams (of all sizes) into a Hilbert space $\mathcal{H} = \mathrm{Span} \{ \ket{\mu}: \text{partitions }\mu \} $.
Our ``basis states" are Young diagrams\footnote{We represent this young diagram as the product of creation and annihilation operators acting on the Hilbert space spanned by young diagrams by adding and removing squares.  This is related to the ``Frobenius coordinates" of partitions.  The resulting ``quantum mechanics" is connected to the representations of the Virasoro algebra in String Theory.} $\ket{\mu}$.   
  ``$q^{L_0}$" is a peculiar name for an operator.  \end{paragraph}
\begin{axiom}{} The {\bf counting operator} $q^{L_0}$ is defined by 
\[ \fbox{ $q^{L_0} \ket{\mu} = q^{\# \text{ boxes}}\ket{\mu} =
q^{|\mu|} \ket{\mu}$ }\]
This is the exponent 
of another counting operator.  The observable $L_0$ counts the number of boxes in each diagram $\mu$, so \[L_0\ket{\mu} = \{ \# \text{ of boxes in }\mu\} \times \ket{\mu}\] 
The operators $L_0, q^{L_0}$ act diagonally on the space of young diagrams.   \end{axiom}
\begin{paragraph}{} The interlacing operators are defined by its action on the basis of young diagrams.
\be
\Gamma_- \ket{\mu} &=& \sum_{  \mu \succ \nu } \ket{\nu} \\
\Gamma_+ \ket{\mu} &=& \sum_{ \mu \prec \nu } \ket{\nu}
\ea 
The $\Gamma_{\pm}$ are ``transfer matrices" giving all the possible interlacings $\nu \succ \mu$.  They are adjoint to one another.   \end{paragraph}
\begin{paragraph}{}Using the $\Gamma_{\pm}$ can count every possible 3D partition in a {\it partition function}:\end{paragraph}
\begin{lem}{} 
\[ Z = \bigg\langle \varnothing \bigg| \prod_{t = 0}^\infty q^{L_0} \Gamma_+ \times q^{L_0} \times q^{L_0} \prod_{-\infty}^{-1} \Gamma_- q^{L_0} \bigg| \varnothing \bigg\rangle \]
\end{lem}
\begin{paragraph}{}This product of vertex operators is a translation of Okounkov's recipe for building 3D partitions out of interlacing Young diagrams.\end{paragraph}
\begin{proof}This infinite product can be interpreted as instructions:
\begin{itemize}
	\item start with the empty partition
	\item add some layers to the middle such that $\mu_{t+1} \succ \mu_t$
	\item add some layers to the middle such that $\mu_{t+1} \prec \mu_t$
	\item finish with the empty partition
\end{itemize}
Every 3D partition can be built this way.
\end{proof}
\begin{paragraph}{}At this point, \cite{ORV} indicates to ``commute the $q^{L_0}$ operators to the outside":\end{paragraph}
\begin{lem}\[  Z = \bigg\langle \varnothing \bigg| \prod_{n > 0} \Gamma_+ ( q^{n - \frac{1}{2}} ) 
\prod_{m \geq 0} \Gamma_- ( q^{-m - \frac{1}{2}} ) \bigg| \varnothing \bigg\rangle \] \end{lem}
\begin{proof}How does $\Gamma_-$ commute with
$q^{L_0}$?   We're solving $q^{L_0} \Gamma_-  = X q^{L_0}$ for some unknown operator, $X$, but  
\[ X = q^{L_0} \Gamma_- q^{-L_0} \equiv \Gamma_-(q) \] 
so we can interpret $X$ as the ``time-evolved" $\Gamma_-$ operator, $\Gamma_-(q)$, and $\Gamma_- = \Gamma_-(1)$.
\[ \Gamma_-(q)\ket{\mu} = q^{L_0} \Gamma_- q^{-L_0}\ket{\mu} = \sum_{\nu \succ \mu}q^{|\nu |- |\mu| }\ket{\nu} \]
We can also time-evolve the $\Gamma_+$ operator, replacing $\succ$ with $\prec$.
\[ \Gamma_+(q)\ket{\mu} = q^{L_0} \Gamma_+ q^{-L_0}\ket{\mu} = \sum_{\nu \prec \mu}q^{|\nu |- |\mu| }\ket{\nu} \]
Once the $q^{L_0}$ operators are ``commuted out", $q^{L_0}\ket{\varnothing}  = \ket{\varnothing}$.
\[  Z = \bigg\langle \varnothing \bigg| \prod_{n > 0} \Gamma_+ ( q^{n - \frac{1}{2}} ) 
\prod_{m \geq 0} \Gamma_- ( q^{-m - \frac{1}{2}} ) \bigg| \varnothing \bigg\rangle \]
The $\frac{1}{2}$'s come from splitting the middle $q^{L_0}$ operator into $q^{\frac{L_0}{2}} \cdot q^{\frac{L_0}{2}}$.
\end{proof}
\begin{paragraph}{}To finish the calculation, recall the commutation relation in Kac' book: \cite{K}.\end{paragraph}
\begin{prop}
\[ \Gamma_+(q^m) \Gamma_-(q^{-n} ) = \frac{1}{1 - q^{m+n}} \Gamma_-(q^{-n}) \Gamma_+ (q^m) \]
\end{prop}
\begin{paragraph}{}
This is proven from a definition of $\Gamma_{\pm}$ using vertex operators or using infinite dimensional Lie algebras.  Let's use interlacing instead.\end{paragraph}
\begin{proof} Let $a = q^m,b=q^{-n}$.  These products are infinite sums over two young tableau valued degrees of freedom.
\be  \Gamma_+(a) \Gamma_-(b^{-1} ) \ket{\mu} &=& \sum_{\nu \succ \mu  } \sum_{ \nu \succ \mu_1 }
a^{ |\mu_1| - |\nu|  }b^{|\mu| - |\nu| } \ket{\mu_1} \\
\Gamma_-(a) \Gamma_+(b^{-1} ) \ket{\mu} &=& \sum_{\nu \prec \mu  } \sum_{ \nu \prec \mu_1 }
a^{ |\mu_1| - |\nu|  }b^{|\mu| - |\nu| } \ket{\mu_1}
\ea
We can show two operators are proportionate by comparing their matrix elements.
\[ (1 - ab)\bra{\mu'} \Gamma_+(a) \Gamma_-(b^{-1} ) \ket{\mu} = \bra{\mu'} \Gamma_-(b^{-1}) \Gamma_+ (a)\ket{\mu} \]
This reduces the degrees of freedom to just $\mu_1$.

What does this sum look like when $\ket{\mu_0} = \ket{\varnothing}$?
\[\Gamma_+(a) \Gamma_-(b^{-1} ) \ket{\varnothing} = \sum_{\nu \succ \varnothing  , \nu \succ \mu_1 }
a^{ |\mu_1| - |\nu|  }b^{|\varnothing| - |\nu| } \ket{\mu_1} =  \sum_{\nu \geq \mu_1 \geq 0 }
a^{|\mu_1| - |\nu|  }b^{ - |\nu| } \ket{\mu_1}  = \sum_{\mu_1 \geq 0}  \frac{ b^{ -|\mu_1|}\ket{\mu_1}}{1 - 1/ab}\]
$\mu_1,\mu_2$ can have only one part, so they are numbers.
In the other sum $\varnothing \succ \mu_1 \prec \nu $, so $\mu_1 = \varnothing$ also.
\[\Gamma_-(a) \Gamma_+(b^{-1} ) \ket{\varnothing} =   \sum_{\mu_1 \prec \varnothing  , \mu_1 \succ \nu }
a^{ |\nu| - |\mu_1|  }b^{|\nu| - |\mu_1| } \ket{\nu}  = \sum_{\nu \geq 0} a^{|\nu|} \ket{\nu} \]
Even though we've only proven it for the vacuum state $\ket{\varnothing}$, we can prove it for any state $\ket{\mu}$.
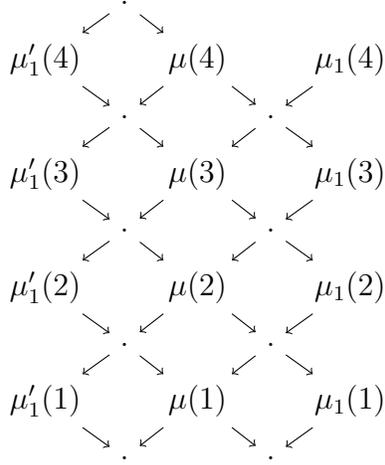
\begin{figure}[h]
\centering
\begin{tikzpicture} 
\matrix (ineqs)[matrix of nodes, row sep = 0.25cm, column sep = 0.25cm, xshift=3cm] { 
 & .&   & & \\
$\mu'_1(4)$&  & $\mu(4)$ &  & $\mu_1(4)$\\
 & .&   & .& \\
$\mu'_1(3)$&  & $\mu(3)$ &  & $\mu_1(3)$\\
 & .&   & .& \\
$\mu'_1(2)$&  & $\mu(2)$ &  & $\mu_1(2)$\\
 & .&   & .& \\ 
$\mu'_1(1)$&  & $\mu(1)$ &  & $\mu_1(1)$\\
 & .&   & .& \\ };

\draw[->] (ineqs-1-2) -- (ineqs-2-1);
\draw[->] (ineqs-3-2) -- (ineqs-4-1);
\draw[->] (ineqs-5-2) -- (ineqs-6-1);
\draw[->] (ineqs-7-2) -- (ineqs-8-1);
\draw[->] (ineqs-2-1) -- (ineqs-3-2);
\draw[->] (ineqs-4-1) -- (ineqs-5-2);
\draw[->] (ineqs-6-1) -- (ineqs-7-2);
\draw[->] (ineqs-8-1) -- (ineqs-9-2);

\draw[->] (ineqs-1-2) -- (ineqs-2-3);
\draw[->] (ineqs-3-2) -- (ineqs-4-3);
\draw[->] (ineqs-5-2) -- (ineqs-6-3);
\draw[->] (ineqs-7-2) -- (ineqs-8-3);
\draw[->] (ineqs-2-3) -- (ineqs-3-2);
\draw[->] (ineqs-4-3) -- (ineqs-5-2);
\draw[->] (ineqs-6-3) -- (ineqs-7-2);
\draw[->] (ineqs-8-3) -- (ineqs-9-2);

\draw[->] (ineqs-3-4) -- (ineqs-4-3);
\draw[->] (ineqs-5-4) -- (ineqs-6-3);
\draw[->] (ineqs-7-4) -- (ineqs-8-3);
\draw[->] (ineqs-2-3) -- (ineqs-3-4);
\draw[->] (ineqs-4-3) -- (ineqs-5-4);
\draw[->] (ineqs-6-3) -- (ineqs-7-4);
\draw[->] (ineqs-8-3) -- (ineqs-9-4);

\draw[->] (ineqs-3-4) -- (ineqs-4-5);
\draw[->] (ineqs-5-4) -- (ineqs-6-5);
\draw[->] (ineqs-7-4) -- (ineqs-8-5);
\draw[->] (ineqs-2-5) -- (ineqs-3-4);
\draw[->] (ineqs-4-5) -- (ineqs-5-4);
\draw[->] (ineqs-6-5) -- (ineqs-7-4);
\draw[->] (ineqs-8-5) -- (ineqs-9-4);

\end{tikzpicture}
\caption{ $\nu'(5)$ can be any number $\geq \mu(4)$.}
\end{figure}

We are trying to compare the two kinds of betweenness $\mu_1 \succ \nu \prec \mu$ 
and $\mu_1' \prec \nu' \succ \mu$.  Observe  \[ \{ \mu_1 : \exists \nu \text{ such that }\mu_1 \succ \nu \prec \mu\} = \{ \mu_1' : \exists \nu' \text{ such that }\mu_1' \prec \nu' \succ \mu\} \]
However, once we fix $\mu_2$, $\nu'$ has an extra degree of freedom -  $\nu'(\mathrm{len}(\mu)+1) \geq \nu(len(\mu))= \mu_{max}$ - which accounts for the $1 - ab$ factor.
\end{proof}
\begin{paragraph}{} Commuting each $\Gamma_+$ across infinitely many $\Gamma_-$'s, a factor of $\frac{1}{1 - q^n}$ appears $n$ times for each $n$.  After that, $\Gamma_-\ket{\varnothing} = \ket{\varnothing}$ and it's adjoint $\bra{\varnothing}\Gamma_+ = \ket{\varnothing}$, so that $Z = \prod_{n > 0} (1 - q^n)^{-n}$. \end{paragraph}
\newpage{\bf Convergence of the Infinite Product}
\begin{paragraph}{} This vertex operator proof of McMahon's formula is not rigorous since it depends on infinitely many commutations and is a formal q-series identity.  Maybe we can specify how these series converge.  Our approach is to replace $\infty$ with $L$ where the products are finite and let $L$ tend to infinity.
\begin{prop}\[  Z = \bigg\langle \varnothing \bigg| \prod_{n > 0}^L \Gamma_+ ( q^{n - \frac{1}{2}} ) 
\prod_{m \geq 0}^L \Gamma_- ( q^{-m - \frac{1}{2}} ) \bigg| \varnothing \bigg\rangle = \prod_{0 \leq m, n\leq L} 
\frac{1}{1 - q^{m+n}} \] \end{prop}
 Mod $q^L$ all of the factors are asymptotically, $1$. So we have a well-defined product.
 \[ \prod_{0 \leq m, n\leq L} 
\frac{1}{1 - q^{m+n}}  = \prod_{n \geq 0} 
\frac{1}{(1 - q^{n})^n} ( \mod q^L )  \]
This is convergence in a ``q-adic" topology on the ring of power series, $\CC[[q]]$.  The ideals $(q^k)$ correspond to the germ at $q =0$.  
  \end{paragraph}

\begin{paragraph}{End} These vertex operators arise in the quantization of the free scalar field on $\Sigma = \RR \times S^1$, see Chapter 11 of \cite{H}.  In \cite{PS}, these techniques for counting plane partitions are generalized other lattices.  In \cite{ZJ}, it's mentioned the matrix elements of $\Gamma_\pm$ are skew-Schur polynomials from the invariant theory of $S_n$, so there may be a relation to Gelfand-Tsetlin patterns.  There are many $q$-series formulas whose proofs may be simplified using these interlacing operators.\end{paragraph}
\newpage \section{Appendix: Partitions and 3D Partitions}
\begin{axiom}A {\bf partition} is a decomposition
of natural numbers into ``parts" , $n = n_1 + n_2 + \dots n_k$, written in 
decreasing order, $n_1 \geq n_2 \geq \dots \geq n_k$.\end{axiom}
\begin{paragraph}{}  
The partition $8 = 5 + 2 + 1$ can be drawn as a Young diagram.:
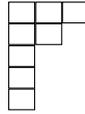
\[  \fig \] 
Plane partitions are the 3D analog of integer partitions. \end{paragraph}
\begin{axiom} A {\bf 3D partition} is an Young diagram with integer 
entries weakly decreasing 
in both directions.
\[ \fg \]
Each point represents the height of the plane partition above each point. \end{axiom}
\begin{paragraph}{} 
This array of numbers is weakly decreasing going to the right and down.
So a plane partition is a matrix $\mu(m,n)$ with $m, n \in \NN$ and
$\mu(x+1, y) \geq \mu(x,y)$, $\mu(x,y+1) \geq \mu(x,y)$ 
and
$\mu(x,y) = 0$ for $x,y > N$ for some $N$.  Outside of a sufficiently large square, all entries are zero.  \end{paragraph}
   
\begin{paragraph}{Acknowledgement} Most of this was done while the author was a graduate student at UC Santa Barbara and a visitor at IPMU in Kashiwa, Japan during Fall 2010.  The author acknowledges David Morrison for his help. \end{paragraph}

\end{document}